\documentclass[12pt]{article}

\usepackage[T1]{fontenc}
\usepackage{amsmath}            
\usepackage{amsfonts}           
\usepackage{amsthm}             
\usepackage{amssymb}            
\usepackage{authblk}

\usepackage{styles/thm_en}      
\usepackage{styles/mymathfonts} 

\newcommand\blfootnote[1]{%
  \begingroup
  \renewcommand\thefootnote{}\footnote{#1}%
  \addtocounter{footnote}{-1}%
  \endgroup
}

\title{Existence and non--uniqueness of global weak solutions to inviscid primitive 
and Boussinesq equations}
\author[1]{Elisabetta Chiodaroli}
\author[2]{Martin Mich\'{a}lek}
  
\affil[1]{EPFL Lausanne

Station 8, CH-1015 Lausanne, Switzerland, \vspace*{0.2cm}}

\affil[2]{
Institute of Mathematics of the Czech Academy of Sciences,
                
\v Zitn\'a 25, 115 67 Praha 1, Czech Republic} 
\date{}

\begin{document}
    \maketitle

    \begin{abstract}
        We consider the initial value problem for the inviscid Primitive and Boussinesq equations 
        in three spatial dimensions. We recast both systems as an abstract Euler-type system and 
        apply the methods of convex integration of De Lellis and Sz\'ekelyhidi to show the existence of infinitely many
        global weak solutions of the studied equations for general initial data. 
        We also introduce an appropriate notion of dissipative solutions and show the existence of suitable initial data 
        which generate infinitely many dissipative solutions.
    \end{abstract}


\section{Introduction}
The Boussinesq equations are used to model the behaviour of oceans. 
Recall that the Boussinesq approximation consists in neglecting changes of density except in
the buoyancy terms and results in a system coupling the incompressible Navier-Stokes equations 
(for an unknown velocity field $\mathbf{u}=\mathbf{u}(t,\mathbf{x}) = (u,v,w)$ and  pressure $p = p(t,\mathbf{x})$) 
with the convection-diffusion equation (for an unknown temperature $\theta=\theta(t,\mathbf{x})$). 
Physically relevant references can be found in \cite{LioTemWang1}. 
We will also consider the effect of the Coriolis force in the form $\Omega \times \mathbf{u}$
for a vector function $\Omega = (\Omega_x,\Omega_y,\Omega_z)$ and neglect the effect of viscosity.
The inviscid Boussinesq equations then read as
\begin{subequations}
    \label{eq:boussinesq}
    \begin{equation}
        \label{eq:bous_mom1}
        \partial_t u + \mathbf{u}\cdot \nabla_{\mathbf{x}} u   +\Omega_y w - \Omega_z v
        +\partial_x p = 0,
    \end{equation}
    \begin{equation}
        \label{eq:bous_mom2}
        \partial_t v + \mathbf{u}\cdot \nabla_{\mathbf{x}} v  -\Omega_x w + \Omega_z u
        +\partial_y p = 0,
    \end{equation}
    \begin{equation}
        \label{eq:bous_mom3}
        \partial_t w + \mathbf{u}\cdot\nabla_{\mathbf{x}} w   +\Omega_x v - \Omega_y u
        +\partial_z p = -\theta,
    \end{equation}
    \begin{equation}
        \label{eq:bous_incom}
        \diver_\mathbf{x} \mathbf{u} = 0,
    \end{equation}
    \begin{equation}
        \label{eq:bous_temp}
        \partial_t \theta + \mathbf{u}\cdot \nabla_{\mathbf{x}} \theta-\lambda_1 (\partial^2_{xx}+\partial^2_{yy})\theta-\lambda_2 \partial^2_{zz} \theta  = 0
    \end{equation}
\end{subequations}
for unknown functions $u$, $v$, $w$, $p$ and $\theta \colon [0,T) \times U \to \BR$. We consider $U\subseteq \BR^3$ an open bounded set. 
The parameters $T$, $\lambda_1$, and $\lambda_2$ 
are given positive real constants without additional restrictions. We will also denote $Q=(0,T)\times U$. 

When modeling the large scale behaviour of oceans or the atmosphere, one spatial scale (vertical) is essentially smaller than the other (horizontal) ones. 
The primitive equations, which are also considered, can be obtained as a formal singular limit of the Boussinesq equations 
in the way that the convective derivative of the vertical velocity coordinate is neglected. The momentum equation for
the vertical velocity component is then replaced by the hydrostatic approximation.
Under the already given notation, the inviscid primitive equations consist of the following system of partial differential equations:
\begin{subequations}
    \label{eq:primeqs}
    \begin{equation}
        \label{eq:mom1}
        \partial_t u + \mathbf{u}\cdot \nabla_{\mathbf{x}} u   - \Omega_z v+ \Omega_y w
        +\partial_x p = 0,
    \end{equation}
    \begin{equation}
        \label{eq:mom2}
        \partial_t v + \mathbf{u}\cdot \nabla_{\mathbf{x}} v   + \Omega_z u - \Omega_x w
        +\partial_y p = 0,
    \end{equation}
    \begin{equation}
        \label{eq:mom3}
        \partial_z p = -\theta,
    \end{equation}
    \begin{equation}
        \label{eq:incom}
        \diver \mathbf{u} = 0,
    \end{equation}
    \begin{equation}
        \label{eq:temp}
        \partial_t \theta + \mathbf{u}\cdot \nabla_{\mathbf{x}} \theta-\lambda_1 (\partial^2_{xx}+\partial^2_{yy})\theta-\lambda_2 \partial^2_{zz} \theta  = 0.
    \end{equation}
\end{subequations}
To complete both systems, we assume the ``no-flow'' boundary condition for $(u,v,w)$ and homogeneous Dirichlet condition for $\theta$:
\begin{subequations}
    \label{eq:bc}
    \begin{equation}
        \label{eq:velocitybc}
        \mathbf{u}(t,\mathbf{x})\cdot \eta(t,\mathbf{x})=0 \quad \mbox{on } (0,T)\times \partial U,
    \end{equation}
    \begin{equation}
        \label{eq:tempbc}
        \theta (t,\mathbf{x})=0 \quad \mbox{on } (0,T)\times \partial U,
    \end{equation}
\end{subequations}
where $\eta$ denotes the exterior normal to the boundary $\partial U$.
Both systems describe the time evolution of $u$, $v$ and $\theta$ and therefore it makes sense to prescribe initial conditions for these quantities. We assume that  
\begin{align}
    \label{eq:ic_uv}
     u(0,\cdot)=u_0, v(0,\cdot)=v_0 \mbox{ and } \theta(0,\cdot)=\theta_0 \quad  \mbox{in } U.
\end{align}
For the Boussinesq equations, we also prescribe the initial vertical velocity
\begin{align}
    \label{eq:ic_w}
     w(0,\cdot)=w_0 \quad  \mbox{in } U.
\end{align}

\begin{rem}
    The main results of the article hold also for other boundary conditions for $\theta$ 
    for which solutions of \eqref{eq:bous_temp} or \eqref{eq:temp} with $\mathbf{u}\in L^{\infty}(Q;\BR^3)$ 
    belong to $\CC(\overbar{Q})$.
\end{rem}

Let us recall known results about the mentioned systems.
The system \eqref{eq:boussinesq} shares many similarities with the Euler system (i.\,e.\, when $\theta = 0$).
In two dimensions, the global well--posedness for regular initial data was established in \cite{chae} 
(see also \cite{danchin} for the recent development).

To the best of our knowledge, the question of the existence of global solutions of \eqref{eq:boussinesq} in 3D remains open. 
We give a positive answer to this question  in the case of weak solutions.
A similar system, namely \eqref{eq:boussinesq} in dimension $2$ with $\lambda_1 = \lambda_2=0$ and without the temperature in \eqref{eq:bous_mom3}, was treated
in \cite{BronziLopez} using a slightly different approach.

Considering three spatial dimensions, there are a few mathematical results connected to the \emph{viscid} (Navier-Stokes-like) 
analogue of system \eqref{eq:primeqs}. 
The local in time existence of regular solutions was presented in \cite{LioTemWang1}, 
where a proof of the global in time existence of weak solutions can also be found.
The existence of global regular solutions under the assumption that initial data are slowly varying in the $z$--variable was proved in \cite{HuTemZia}.
Cao and Titi demonstrated in \cite{CaoLiTitiregular} that the solutions emanating from regular initial data stay regular for all times $t>0$. 
The regularity was also shown in the case of homogeneous Dirichlet boundary conditions in \cite{KukavicaZiane}. 
Very recently, global strong well-posedness in $L^p$ was given in \cite{hieber}.
These results should be put into comparison with the similar Navier-Stokes system for which the question on the global regularity is still open. 
Very recently, it was shown in \cite{CaoTitiblowup} that if one drops the viscous term and considers the system \eqref{eq:primeqs} with suitable 
boundary conditions then a finite time blow-up occurs for some specific regular initial data. Local existence of regular solutions
for inviscid primitive equations in 2D was given in \cite{brenier}. 

A question remaining open is whether there exist global weak solutions for any (suitably regular) initial data for \eqref{eq:primeqs} 
with \eqref{eq:bc}. 
At first glance, there is almost no hope in any kind of positive answer.
The inviscid primitive equations differ notably from the incompressible Euler system. 
In comparison to the Euler system, primitive equations are degenerate with respect to $w$. 
It is known that the system is not hyperbolic and that the boundary value problem is ill-posed for pointwise 
boundary conditions, see \cite{OligerSundstrom} 
and also \cite{RousTemamTrib}. On the other hand, we recall that De Lellis and Sz\'{e}kelyhidi (see e.\,g.\,\cite{DeleSzekAnnals}, 
\cite{DeleSzekArch}) 
extended the possibility to use techniques of convex integration on the Euler system. 
They constructed  infinitely many ``oscillatory" weak solutions 
satisfying even different admissibility criteria, yet exceptionally non-unique.
The aim of this paper is to demonstrate that the inviscid primitive equations also admit such 
oscillatory solutions. 
We will employ the recent refinements of the De Lellis and Sz\'{e}kelyhidi approach,  
carried out in \cite{DonaFeirMarc} and in \cite{ChioFeirKrem} for the Euler-Fourier system
 or in \cite{SavageHutter} for Savage-Hutter model. 

We give the reader the outline of the rest of the article: in Section \ref{sec:main} 
we define the notion of weak solution and formulate the main results. 
In Section \ref{sec:reformulation} we give a reformulation of the given systems into an abstract Euler-type problem. 
In Sections \ref{sec:existence} and \ref{sec:osc_lem_proof} we present the proof of existence of infinitely weak solutions with 
general initial conditions for the abstract problem combining approaches from \cite{DeleSzekAnnals}, \cite{DeleSzekArch}, \cite{Chiodaroli}, \cite{ChioFeirKrem} and \cite{FeirConvexInt}. 
The result is extended in Section \ref{sec:dissipative}, where the existence of some suitable initial data allowing 
for infinitely many dissipative weak solutions is proven. For the reader's convenience, Section \ref{sec:appendix} 
contains some auxiliary results which are employed in the article.

\section{The main results}
\label{sec:main}
We will denote by $\CC([0,T]; X_{w})$ the set of continuous functions from $[0,T]$ with values in 
a Banach space $X$ equipped with the weak topology. 
For $B\subseteq \BR^d$ open we denote $\CD(B)$  the topological vector space of smooth functions with compact support in $B$ and $\CD'(B)$ its topological dual. 

\subsection{The Boussinesq equations}
We start with introducing the definition of weak solutions to the problem \eqref{eq:boussinesq} supplemented by \eqref{eq:bc} and \eqref{eq:ic_uv}, \eqref{eq:ic_w}. 
\begin{dfn}
    \label{dfn:weaksol_bous} 
    We call the quintet of functions $(u,v,w,p,\theta)$ a \emph{weak solution of the inviscid Boussinesq equations} with 
    \eqref{eq:bc},  \eqref{eq:ic_uv} if
    \begin{itemize}
        \item   $u$, $v$, $w\in \CC([0,T];L^2_w(U))$, $p\in L^1(Q)$ and equations
            \begin{align}
                \label{eq:bous_weakform_u} 
                &\int_{0}^{T}\int_{U}u\partial_t \phi_1\de{\mathbf{x}}\de{t}
                +\int_{0}^{T}\int_{U}u \mathbf{u} \cdot \nabla_{\mathbf{x}}\phi_1\de{\mathbf{x}}\de{t}
                +\int_{U}u_0 (\cdot)\phi_1(0,\cdot)\de{\mathbf{x}}
                \\
                \nonumber
                &\quad\quad+\int_{0}^{T}\int_{U} (-\Omega_y w+\Omega_z v)\phi_1\de{\mathbf{x}}\de{t}+\int_{0}^{T}\int_{U}p \partial_x \phi_1\de{\mathbf{x}}\de{t}=0,
                \\
                \label{eq:bous_weakform_v} 
                &\int_{0}^{T}\int_{U}v\partial_t \phi_2\de{\mathbf{x}}\de{t}
                +\int_{0}^{T}\int_{U}v \mathbf{u} \cdot \nabla_{\mathbf{x}}\phi_2 \de{\mathbf{x}}\de{t}
                +\int_{U}v_0(\cdot) \phi_2(0,\cdot)\de{\mathbf{x}}
                \\
                \nonumber
                &\quad\quad+\int_{0}^{T}\int_{U} (\Omega_x w-\Omega_z u)\phi_2\de{\mathbf{x}}\de{t}+\int_{0}^{T}\int_{U}p \partial_y \phi_2\de{\mathbf{x}}\de{t}
                =0,
                \\
                \label{eq:bous_weakform_w} 
                &\int_{0}^{T}\int_{U}w\partial_t \phi_3\de{\mathbf{x}}\de{t}
                +\int_{0}^{T}\int_{U}w   \mathbf{u} \cdot \nabla_{\mathbf{x}}\phi_3\de{\mathbf{x}}\de{t}
                +\int_{U}w_0 (\cdot)\phi_3(0,\cdot)\de{\mathbf{x}}
                \\
                \nonumber
                &\quad\quad+\int_{0}^{T}\int_{U} (-\Omega_x v+\Omega_y u)\phi_3\de{\mathbf{x}}\de{t}+\int_{0}^{T}\int_{U}p \partial_z \phi_3\de{\mathbf{x}}\de{t}
                  =  \int_{0}^{T}\int_{U}\theta \phi_3  \de{\mathbf{x}}\de{t}
            \end{align}
        are satisfied for any $\phi_1$, $\phi_2$, $\phi_3 \in \CD([0,T)\times U)$,
        
        \item $\mathbf{u}\chi_{Q}$ solves \eqref{eq:bous_incom} in $\CD'((0,T)\times \BR^3)$, i.\,e.\,
            \begin{equation}
                \label{eq:bous_weak_incom}
                \int_{0}^T \int_{U}\mathbf{u}\cdot\nabla_{\mathbf{x}}\phi\de{\mathbf{x}}\de{t} = 0 \mbox{ for every $\phi \in \CD((0,T)\times \BR^3)$},  
            \end{equation}
        \item $\theta\in W^{1,p}\left((0,T); L^p(U)\right)\cap L^p\left((0,T);W^{2,p}(U)\cap W^{1,p}_0(U)\right)$ for a $p\in (1,\infty)$ and \eqref{eq:bous_temp} holds almost everywhere in $Q$ and $\theta(0,\cdot)=\theta_0(\cdot)$ in the sense of time traces.
    \end{itemize}
\end{dfn}

\begin{thm}
    \label{thm:bous_main1}
    Let $T>0$, $U$ be a bounded open set with $\partial U\in \CC^{2}$, $\mathbf{u}_0 \in L^{\infty}(U;\BR^3)\cap \CC(U;\BR^3)$ 
    with $\diver_\mathbf{x} \mathbf({u}_0\chi_{U}) =0$ in the sense of distributions, 
    $\theta_0 \in L^{\infty}(Q) \cap \CC^{2}(Q)$ and $\Omega=\Omega(\mathbf{x}) \in L^{\infty}(U;\BR^3)$. 
    Then there exist infinitely many weak solutions to the Boussinesq equations in the sense of Definition \ref{dfn:weaksol_bous}.
\end{thm}

For the Boussinesq equations, the total energy is defined as the sum of the kinetic
and potential energy:
    \begin{equation*}
        E_{Bous}(t) = \int_{U}\frac{1}{2}|\mathbf{u}(t,\mathbf{x})|^2+z\theta(t,\mathbf{x}) \de{\mathbf{x}},
    \end{equation*}
see also \cite{winters}.
Referring to \cite{diaz}, we recall that the Boussinesq equations violate the principle of conservation of total energy.
The quantity $E_{Bous}$ would be conserved if the heat dissipation were neglected.

\begin{dfn}
    We say that a weak solution of the Boussinesq equations satisfies 
    the \emph{strong energy inequality} if $E_{Bous}(t)$ is non-increasing on $[0,T)$. We also call such solutions \emph{dissipative}.
\end{dfn}
   
Let us mention that the weak solutions given by Theorem \ref{thm:bous_main1} are violating the strong energy inequality, particularly
    \[
        \liminf_{t\to 0^+} E_{Bous}(t)>E_{Bous}(0).
    \]
   
\begin{thm}
    \label{thm:bous_main2}
    Let $\Omega \in L^{\infty}(U;\BR^3))$ and $\theta_0\in L^{\infty}(U)\cap \CC^{2}(U)$. 
    Then there exists $\mathbf{u}_0\in L^{\infty}(U;\BR^3)$ for which we can find 
    infinitely many weak dissipative solutions of the Boussinesq equations emanating from $\mathbf{u}_0$.
\end{thm}

\subsection{The primitive equations}
Analogously, we present the definition of the weak solutions to \eqref{eq:primeqs}:
\begin{dfn}
    \label{dfn:weaksol} 
    We call the quintet of functions $(u,v,w,p,\theta)$ a \emph{weak solution of the inviscid primitive equations} 
    with \eqref{eq:bc},  \eqref{eq:ic_uv} if
    \begin{itemize}
        \item   $\mathbf{u}=(u,v,w) \in L^{2}(Q;\BR^3)$, $u$, $v\in \CC([0,T];L^2_w(U))$, $p\in L^1(Q)$, $\partial_z p \in L^1(Q)$ and equations
                \eqref{eq:bous_weakform_u} and \eqref{eq:bous_weakform_v} 
        are satisfied for any $\phi_1$, $\phi_2 \in \CD([0,T)\times U)$,        
        \item $\mathbf{u}\chi_{Q}$ solves \eqref{eq:bous_weak_incom},
        
        \item $\theta\in W^{1,p}\left((0,T); L^p(U)\right)\cap L^p\left((0,T);W^{2,p}(U)\cap W^{1,p}_0(U)\right)$ for a $p\in (1,\infty)$ and \eqref{eq:temp} holds almost everywhere in $Q$ and $\theta(0,\cdot)=\theta_0(\cdot)$ in the sense of time traces.
    
        \item equation \eqref{eq:mom3} holds for the weak derivative of $p$ almost everywhere in $Q$.
        \end{itemize}
\end{dfn}

As we will see, the problem of finding weak solutions of the primitive equations is highly underdetermined. Let us fix a function $p$ and supplement the system by an equation describing the evolution of $w$:
\begin{align}
    \label{eq:suplw}
    \partial_t w + \mathbf{u}\cdot\nabla_{\mathbf{x}} w +\Omega_x v-\Omega_y u+\partial_z p=0.
\end{align}

The equations \eqref{eq:mom1}, \eqref{eq:mom2} and \eqref{eq:suplw} can be recast in the usual vector form
\begin{align*}
    \partial_t \mathbf{u} + \diver_\mathbf{x}(\mathbf{u}\otimes \mathbf{u})+\Omega \times \mathbf{u}+\nabla_{\mathbf{x}} p = 0,
\end{align*}             
where $\mathbf{u}=(u,v,w)$. We use the notion \emph{extended primitive equations} for the system \eqref{eq:primeqs} coupled with 
\eqref{eq:suplw} together with an additional initial condition for $w$. 
For the sake of completeness, we present the definition of the corresponding weak solution.
\begin{dfn}
    We call $(\mathbf{u},p,\theta)$ a \emph{weak solution of the extended primitive equations} with \eqref{eq:bc}, \eqref{eq:ic_uv},
    \eqref{eq:ic_w}  if 
    \begin{itemize}
        \item $(\mathbf{u},p,\theta)$ is a weak solution of the inviscid primitive equations with \eqref{eq:bc},  \eqref{eq:ic_uv} 
        \item $w\in \CC([0,T];L^2_{w}(U))$ and 
            \begin{align}
                \label{eq:weakform_w} 
                &\int_{0}^{T}\int_{U}w\partial_t \phi_3\de{\mathbf{x}}\de{t}
                +\int_{0}^{T}\int_{U}w   \mathbf{u} \cdot \nabla_{\mathbf{x}}\phi_3\de{\mathbf{x}}\de{t}
                +\int_{U}w_0 (\cdot)\phi_3(0,\cdot)\de{\mathbf{x}}
                \\
                \nonumber
                &\quad\quad+\int_{0}^{T}\int_{U} (-\Omega_x v+\Omega_y u)\phi_3\de{\mathbf{x}}\de{t}+\int_{0}^{T}\int_{U}p \partial_z \phi_3\de{\mathbf{x}}\de{t}
                  =  0
            \end{align}
        is satisfied for any $\phi_3 \in \CD([0,T)\times U)$.
    \end{itemize}
\end{dfn}
\begin{rem}
Because of the well-posedness result, it seems to be unreasonable to work with the extended primitive equations.
However, all the results following the approach of De Lellis and Sz\'ekelyhidi (see e.\,g.\,\cite{DeleSzekAnnals}, \cite{DeleSzekArch}, \cite{DonaFeirMarc} or \cite{ChioFeirKrem}) are foreshadowing 
that weak formulations of inviscid problems in fluid dynamics are surprisingly highly underdetermined. The main results of this paper, namely Theorem \ref{thm:main1}, \ref{thm:main2} and Corollary \ref{cor:main1}, 
are in agreement with this observation.
\end{rem}

\begin{thm}
    \label{thm:main1}
    Let $T>0$, $U$ be a bounded open set with $\partial U\in \CC^{2}$, $\mathbf{u}_0 \in L^{\infty}(U;\BR^3)\cap \CC(U;\BR^3)$
     with $\diver_\mathbf{x} \mathbf({u}_0\chi_{U}) =0$ in the sense of distributions, 
     $\theta_0 \in L^{\infty}(Q) \cap \CC^{2}(Q)$ and $\Omega\in L^{\infty}(U;\BR^3)$. 
    Then there exist infinitely many weak solutions to the extended primitive equations with 
    \eqref{eq:bc}, \eqref{eq:ic_uv} and \eqref{eq:ic_w}.
\end{thm}

\begin{cor}
    \label{cor:main1}
    Let $\theta_0 \in L^{\infty}(Q) \cap \CC^{2}(Q)$, $\Omega=\Omega(\mathbf{x})\in L^{\infty}(U;\BR^3)$ 
    and let $u_0$, $v_0 \in L^{\infty}(U)\cap \CC(U)$ 
    be such that exists $w_0 \in L^{\infty}(U)\cap \CC(U)$ satisfying 
    \begin{equation}
      \label{eq:extra_cond}
        \diver_\mathbf{x} ((u_0,v_0,w_0)\chi_U)=0
    \end{equation}
    in the sense of distributions on $\BR^3$. Then there exist infinitely many weak 
    solutions to the primitive equations with \eqref{eq:bc} and \eqref{eq:ic_uv}.
\end{cor}

\begin{rem}
    The technical assumption on $u_0$ and $v_0$ is needed only because we are considering boundary conditions
    \eqref{eq:velocitybc}. If we took $U=\BT^3$ then the additional condition leading to \eqref{eq:extra_cond}
    would be $\partial_x u_0+\partial_y v_0 \in L^{\infty}(U)\cap \CC(U)$.
\end{rem}

To the best of our knowledge, there are no a priori estimates on $(u,v,w)$ in the case of inviscid primitive equations.
Still, it is possible to find initial data for which there exist infinitely 
many weak solutions of the primitive equations satisfying the conservation of the kinetic energy or which 
are dissipating the mechanical energy. 
Let us define
    \begin{equation*}
        E_{Prim}(t)= \int_{U}\frac{1}{2}\left(|u(t,\mathbf{x})|^2+|v(t,\mathbf{x})|^2+|w(t,\mathbf{x})|^2\right)
               \de{\mathbf{x}}.
    \end{equation*}
    
\begin{dfn}
    We say that a weak solution of the extended primitive equations  
    or primitive equations satisfies the \emph{strong energy inequality} if $E_{Prim}(t)$ is non-increasing on $[0,T)$. 
    We also call such solutions \emph{dissipative}.
\end{dfn}

    Similarly to the previous section the weak solutions given by Theorem \ref{thm:main1} are violating the strong energy inequality, particularly
    \[
        \liminf_{t\to 0^+} E_{Prim}(t)>E_{Prim}(0).
    \]
                                                                    
\begin{thm}
    \label{thm:main2}
    There exists $\mathbf{u}_0\in L^{\infty}(U;\BR^3)$ for which we can find infinitely 
    many weak dissipative solutions of the extended primitive equations emanating from $\mathbf{u}_0$.
\end{thm}
\begin{rem}
    To obtain the largest possible space for initial temperatures in which the given method holds, we can apply the theory of maximal regularity for parabolic equations (see e.\,g.\,\cite{amann}). Particularly, $\theta_0$  can be taken arbitrarily from the interpolation space $[L^p,W^{2,p}]_{\alpha}$ for a suitable $p\in (1,\infty)$ and $\alpha\in(0,1)$.
\end{rem}

\section{Abstract Euler-type system}
\label{sec:reformulation}
To use the techniques from 
\cite{DeleSzekAnnals}, \cite{DeleSzekArch}, we will follow \cite{SavageHutter}, \cite{FeirConvexInt} 
and reformulate the Boussinesq and the extended primitive equations 
as an Euler-type equation.
Let us denote by $\mathbf{u}\odot \mathbf{u} = \mathbf{u}\otimes \mathbf{u} - \frac{1}{3}|\mathbf{u}|^2\BI $ 
the traceless part of the symmetric matrix $\mathbf{u}\otimes \mathbf{u}$. 
We will introduce operators $\BH\colon L^{\infty}(Q;\BR^3) \to L^{1}(Q;\BR^{3\times 3}_{0,sym})$,
 $\Pi\colon L^{\infty}(Q;\BR^3) \to L^{1}(Q)$ and consider the abstract Euler-type system
\begin{subequations}
    \label{eq:abs}
     \begin{equation}
        \label{eq:abs_mom}
        \partial_t \mathbf{u} + \diver_{\mathbf{x}}\left(\mathbf{u}\odot \mathbf{u}+\BH(\mathbf{u})\right) + \nabla_{\mathbf{x}}\left( \Pi[\mathbf{u}]+\frac{1}{3}|\mathbf{u}|^2\right)=0
        \mbox{ in } Q,
     \end{equation}
     \begin{equation}
        \label{eq:abs_incom}
        \diver_{\mathbf{x}} (\mathbf{u})=0 \mbox{ in } Q,
     \end{equation}
     \begin{equation}
        \label{eq:abs_bound_cond}
        \mathbf{u}\cdot \eta=0 \mbox{ on } (0,T)\times \partial U,     
    \end{equation}
     \begin{equation}
        \label{eq:abs_init_cond}
        \mathbf{u}(0)=\mathbf{u}_0 \mbox{ in } U,    
    \end{equation}
\end{subequations}

For the sake of completeness, we add a definition of weak solutions of \eqref{eq:abs}:
\begin{dfn}
    We say that $\mathbf{u}\colon Q\to\BR^{3}$ is a weak solution of the abstract Euler system \eqref{eq:abs} if
    \begin{itemize}
        \item $\mathbf{u}\in \CC([0,T];L_w^2(U))$,
        \item $\mathbf{u}$ satisfies \eqref{eq:abs_mom} in $\CD'(Q)$,
         \item $\mathbf{u}\chi_U$ satisfies \eqref{eq:abs_incom} in $\CD'((0,T)\times\BR^3)$,
        \item $\mathbf{u}(0)=\mathbf{u}_0$.
    \end{itemize}
\end{dfn}

\begin{thm}
\label{thm:abstr_exist}
  Let $T>0$, 
  $\mathbf{u}_0 \in L^{\infty}(U;\BR^3)\cap \CC(U;\BR^3)$ 
    with $\diver_\mathbf{x} \mathbf({u}_0\chi_{U}) =0$ in the sense of distributions.
    Assume that the $\BH$ and $\Pi$ have the following properties:
    \begin{itemize}
      \item $\BH$ is continuous from  $\CC([0,T];L^q_{w}(U))$ to $\CC(\overbar{Q};\BR^{3\times 3}_{0,sym})$
      and mapping bounded sets to bounded sets (with respect to the mentioned topologies).
      \item $\Pi$ is continuous from  $\CC([0,T];L^q_{w}(U))$ to $\CC(\overbar{Q})$
      and there exists $\overbar{\Pi}\in \BR$ such that 
      \begin{equation}
        \label{eq:bound_pressure}
        \Pi[\mathbf{u}]<\overbar{\Pi} \quad \mbox{for every }\mathbf{u}\in L^{\infty}(Q;\BR^3).
      \end{equation}
      \item  For $\mathbf{u}$, $\mathbf{w}\in L^{\infty}(Q;\BR^3)$  
      with $\supp \mathbf{w} \subseteq (\tau,T)\times \overbar{U}$
      \begin{equation}
        \label{eq:nonanticip}
          \Pi[\mathbf{u+w}]=\Pi[\mathbf{u}], \ \BH[\mathbf{u+w}]=\BH[\mathbf{u}], \
          \mbox{almost everywhere in $(0,\tau)\times U$}.
      \end{equation}
    \end{itemize}
    Then there exist infinitely many weak solutions to \eqref{eq:abs} which moreover satisfy 
\begin{equation} \label{eq:energy constraint}
 \frac{3}{2} \Pi[\mathbf{u}](t, \mathbf{x})+\frac{1}{2}|\mathbf{u}(t, \mathbf{x})|^2= \frac{3}{2} Z(t) \quad \text{for every}\; t\in (0,T), \; \text{almost everywhere in }U,
\end{equation}
  for any function $Z(t)$ continuous on $[0,T]$ satisfying $\sup_{t\in [0,T]} Z(t)>\overbar{\Pi}$.
\end{thm}

In the first part of this Section, we will show that Theorem \ref{thm:bous_main1} and \ref{thm:main1} 
follow directly from Theorem \ref{thm:abstr_exist} after suitable
choices of $\BH$ and $\Pi$. 
In the second part, we will present the notion of subsolution
for \eqref{eq:abs} and important properties of the set of subsolutions. 

\subsection{Reformulation of the Boussinesq equations}
Assume that $\theta_0$ and $\Omega$ comply with the assumptions of 
Theorem \ref{thm:bous_main1}. Using the classical theory of parabolic equations (see e.\,g.\,Lemma \ref{lem:parab_regul}) 
we can define an  operator $\Theta=\Theta[\mathbf{u}]$  from $L^{\infty}(Q;\BR^3)$ to $\CC([0,T]\times \overbar{U})$ such 
that $\theta=\Theta[\mathbf{u}]$ solves \eqref{eq:bous_temp} 
in the sense of Definition \ref{dfn:weaksol_bous}.  The operator
$\mathbf{u}\mapsto \Theta[\mathbf{u}]$ is continuous from $\CC([0,T];L^q_{w}(U))$
to $\CC([0,T]\times \overbar{U})$ (for $q$ large enough).

Using Corollary \ref{cor:tosym}, we obtain the existence of a 
linear operator $\BH_{Bous}=\BH_{Bous}[\mathbf{u}]$
from $L^{\infty}(Q;\BR^3)$ to $\CC(\overbar{Q};\BR^{3\times 3}_{0,sym})$ such that
\[
    \diver_{\mathbf{x}}(\BH_{Bous}[\mathbf{u}])=\Omega \times \mathbf{u}+\begin{pmatrix} 0 \\ 0 \\ \Theta[\mathbf{u}] \end{pmatrix}-\nabla\left(\frac{2}{3}z\Theta[\mathbf{u}]\right).
\]
Define $\Pi_{Bous}[\mathbf{u}] = \frac{2}{3}z\Theta[\mathbf{u}]$.
The operator $\Pi_{Bous}$ is continuous from $\CC([0,T];L^q_{w}(U))$ to $\CC(\overbar{Q})$ and $\BH_{Bous}$ is continuous from $\CC([0,T];L^q_w(U;\BR^3))$ to $\CC(\overbar{Q})$ for any $q >3$. 
Both operators are mapping bounded sets from  $L^{\infty}(Q)$ on bounded sets in $\CC(\overbar{Q})$.
Using the maximum principle for \eqref{eq:temp}, see Lemma \ref{lem:parab_regul}, we obtain
\begin{equation*}
    |\Pi_{Bous}[u](t,\mathbf{x})|\leq \frac{2}{3} \|z\cdot\theta_0\|_{L^{\infty}(U)}<\infty \mbox{ for every $(t,\mathbf{x})\in Q$.}
\end{equation*}
The condition \eqref{eq:nonanticip} holds from the defition of the operators. Indeed,
$\Theta \colon L^{\infty}(Q;\BR^3)$ is a unique solution of an evolutionary equation
and $\mathbf{v}=0$ is the only solution of \eqref{eq:lame} with $\mathbf{g}=0$ with zero boundary conditions.  
If $\mathbf{u}$ is a weak solution of \eqref{eq:abs} then 
the triplet $(\mathbf{u},p,\theta)=(\mathbf{u}, 0 , \Theta[\mathbf{u}] )$
is a weak solution of \eqref{eq:boussinesq}, hence Theorem \ref{thm:bous_main1} 
is a corollary of Theorem \ref{thm:abstr_exist}.

\subsection{Reformulation of the extended primitive equations}
Assume that $\theta_0$, $\mathbf{u}_0$, $P$ and $\Omega$ comply with the assumptions of Theorem \ref{thm:main1}. 
For a given $\mathbf{u}$ we can extend $\Theta[\mathbf{u}]$ continuously  with respect to space on $[0,T]\times \BR^3$. 
As $U$ is bounded, one can define an extension $\Theta[\mathbf{u}]$ such that it has compact support. 
Let us take arbitrary function $P \in \CC([0,T]\times \BR^2)$.
The function $p\colon Q \to \BR$ defined by
\[
    p(t,x,y,z)=P(t,x,y)-\int_{-\infty}^z \Theta[\mathbf{u}](t,x,y,s)\de{s}
\]
satisfies \eqref{eq:mom3} and the operator $\Gamma_{Prim}=\Gamma_{Prim}[\mathbf{u}]\colon \mathbf{u} \mapsto p$ 
maps functions from $L^{\infty}(Q;\BR^3)$ to $\CC(\overbar{Q})$. 
The operator $\Gamma_{Prim}$ is continuous from $\CC([0,T];L^q_{w}(U)))$ to $\CC(\overbar{Q})$.
Using Corollary \ref{cor:tosym}, we obtain existence of a linear operator $\BH_{Prim}=\BH_{Prim}[\mathbf{u}]$
from $L^{\infty}(Q;\BR^3)$ to $\CC(\overbar{Q};\BR^{3\times 3}_{0,sym})$ such that
\[
    \diver_{\mathbf{x}}(\BH_{Prim}[\mathbf{u}])=\Omega \times \mathbf{u} + \nabla \Gamma_{Prim}[\mathbf{u}].
\]
 Similarly to the Boussinesq equations, $\BH_{Prim}$ is continuous from $\CC([0,T];L^q_w(U;\BR^3))$ to $\CC(\overbar{Q})$ for any $q >3$. 
Both operators are mapping bounded sets from  $L^{\infty}(Q)$ on bounded sets in $\CC(\overbar{Q})$.
For any weak solution $\mathbf{u}$ of \eqref{eq:abs}
the triplet $(\mathbf{u},p,\theta)=(\mathbf{u}, \Gamma_{Prim} , \Theta[\mathbf{u}] )$
is a weak solution of the extended primitive equations. Therefore also Theorem \ref{thm:main1} 
is a corollary of Theorem \ref{thm:abstr_exist}.

\subsection{Subsolutions for the abstract Euler equation}

Let $\mathbf{u}_0$ comply with the assumptions of Theorem \ref{thm:abstr_exist}.
Observe that the abstract system is invariant with respect to adding a continuous
function $Z=Z(t)$ to the pressure term $\Pi[\mathbf{u}]$. Moreover, 
$\tilde{\Pi}[\mathbf{u}] = {\Pi}[\mathbf{u}]+ Z$ satisfies the same qualitative properties
as $\Pi$ in Theorem \ref{thm:abstr_exist}.

Let us fix $Z=Z(t)$ continuous on $[0,T]$ such that
  
  \[
    \Pi[\mathbf{v}(t,\mathbf{x})] <Z(t)    \quad \mbox{for }
    (t,\mathbf{x})\in Q
  \]
for any $\mathbf{v}\in L^{\infty}(Q;\BR^3)$. Such function $Z$ exists due to the boundedness of $\Pi$. 
We  restrict our attention only on the so-called pressureless case, i.\,e. when solutions are satisfying 
\[
    \Pi[\mathbf{u}]+\frac{1}{3}|\mathbf{u}|^2-Z(t)=0 \quad \mbox{in } Q.
\]

Mimicking the strategy of De Lellis and Sz\'{e}kelyhidi we recast the abstract system into 
a linear system supplemented by implicit constitutive (possibly non-algebraic) relations:
\begin{subequations}
    \label{eq:nonlin_separ}
    \begin{equation}
        \label{eq:linear}
        \partial_t \mathbf{u} + \diver_{\mathbf{x}}\BV=0,    
    \end{equation}
    \begin{equation}
        \label{eq:incompress}
        \diver_{\mathbf{x}} \mathbf{u}=0,
    \end{equation}
    \begin{equation}
        \BV = \mathbf{u}\odot \mathbf{u}+\BH(\mathbf{u}),  
    \end{equation}  
    \begin{equation}
        \frac{1}{2}|\mathbf{u}|^2=\frac{3}{2}\left(Z(t)-\Pi[\mathbf{u}] \right).
    \end{equation}
\end{subequations}

To introduce a suitable notion of subsolution we put
\[
    \overbar{e}[\mathbf{u}] = \frac{3}{2}\left( Z(t)-\Pi[\mathbf{u}] \right) 
\]
and
\[
    e(\mathbf{u}, \BV) = \frac{3}{2}\lambda_{\max}[\mathbf{u}\otimes \mathbf{u} + \BH[\mathbf{u}]-\BV],
\]
where $\lambda_{\max}(\BU)$ denotes the maximal eigenvalue of $\BU\in \BR^{3\times 3}_{sym}$. 
One has for any $\mathbf{v}\in \BR^3$ and $\BU \in \BR^{3\times 3}_{0,sym}$ the following inequality
\begin{equation}
    \label{eq:bound_v}
    \frac{1}{2}|\mathbf{v}|^2\leq \frac{3}{2}\lambda_{\max}\left( \mathbf{v}\otimes \mathbf{v} +\BU    \right)
\end{equation}
and the equality holds if and only if
\begin{equation}
    \label{eq:sym_matr_eq}
    \BU = \mathbf{v}\otimes \mathbf{v} - \frac{1}{3}|\mathbf{v}|^2\BI.
\end{equation}
It is possible to estimate $\BU$ by the means of $|e(\mathbf{v},\BU)|$, particularly
\begin{equation}
    \label{eq:bound_U}
    |\BU|_{\ell^\infty}\leq 2|\lambda_{\min}(\BU)|\leq \frac{4}{3}e(\mathbf{v},\BU).
\end{equation}
    
Analogously to \cite{FeirConvexInt}: 
\begin{dfn}
We call a pair  $(\mathbf{u}, \BV)$ a \emph{subsolution of the abstract Euler system} (or briefly a \emph{subsolution}) if
\begin{enumerate}
    \item $\mathbf{u}\in  C([0,T];L_w^2(U;\BR^3)) \cap \CC(Q;\BR^3)$ and
    $\BV\in L^{\infty}\cap \CC(Q;\BR^{3\times 3}_{0,sym})$,
    \item  the pair $(\mathbf{u}, \BV)$ satisfies \eqref{eq:linear} in the sense of distributions on $Q$ and $\mathbf{u}\chi_U$ solves \eqref{eq:incompress} in the sense of distributions on $(0,T)\times \BR^3$,
    \item $\mathbf{u}(0)=\mathbf{u}_0$,
    \item \label{enum:subsol_ineq} for every $0<\tau<T$
    $\essinf_{t\in (\tau,T),\,\mathbf{x}\in U}\left( \overbar{e}[\mathbf{u}](t,\mathbf{x})-e(\mathbf{u}(t,\mathbf{x}),\BV(t,\mathbf{x}))    \right)>0.$
\end{enumerate}
\end{dfn}
We denote $X_0$ the set of all $\mathbf{u}$ for which exist $\BV$ such that $(\mathbf{u},\BV)$ is a subsolution of the abstract Euler type system. 
Let us remark that there exists a constant $E>0$ such that $\overbar{e}[\mathbf{u}]\leq E$ for every $\mathbf{u}\in X_0$.
Observe that \eqref{eq:bound_pressure}, \eqref{eq:bound_v} and \eqref{eq:bound_U} imply the boundedness of $X_0$ in $L^{\infty}(Q;\BR^3)$.

We consider for each $\tau\in(0,T\slash 2)$ a negative functional $I_\tau$ on $X_0$ defined by
\[
    I_{\tau}(\mathbf{u}) = \inf_{t\in (\tau, T-\tau)}\int_{U}\frac{1}{2}|\mathbf{u}(t,\mathbf{x})|^2-\overbar{e}[\mathbf{u}(t,\mathbf{x})]\de{\mathbf{x}}.
\]

\begin{lem}
    \label{lem:subsolconv}
    Let $\{(\mathbf{u}_n, \BV_n)\}_{n\in \BN}$ be subsolutions.
    Then there exists a pair $(\mathbf{u},\BV)$ such that for a suitable subsequence (not relabeled)
    \begin{align}
     \label{eq:convergence}
        &\mathbf{u}_n \to \mathbf{u} \mbox{ strongly in } \CC([0,T];L_{w}^2(U; \BR^3)) \mbox{ and weakly-$*$} \in L^{\infty}(Q;\BR^3) \quad
        \\
        &\BV_n \to \BV \mbox{ weakly-$*$ in } L^{\infty}(Q;\BR^{3\times 3}_{0,sym})
    \end{align} 
     holds.  The limit $(\mathbf{u},\BV)$ is satisfying all conditions on subsolutions 
   except condition \ref{enum:subsol_ineq}, where only (nonstrict) inequality holds. 
   Moreover, if 
  \begin{equation}
    \label{eq:maximal_energy}
    I_{\tau}(\mathbf{u})=0 \ \mbox{for each } \tau>0
  \end{equation} 
  then $\mathbf{u}$ is a weak solution of the abstract Euler system \eqref{eq:abs} satisfying \eqref{eq:energy constraint} for suitable functions $Z$.
  
\end{lem}
\begin{proof}
    As $X_0$ consists of functions bounded in $L^{\infty}(Q;\BR^3)$, $\mathbf{u}_n$ resp. $\BV_n$ are also uniformly bounded in $L^{\infty}(Q;\BR^3)$ resp. $L^{\infty}(Q;\BR^{3\times 3}_{0,sym})$. 
    The standard time regularity for weak solutions together with the Arzel\`a-Ascoli theorem implies the existence of $\mathbf{u}$ and $\BV$ such that \eqref{eq:convergence} holds for a suitable subsequence. 
    With respect to Lemma \ref{lem:subsol_limit}, the pair $(\mathbf{u},\BV)$ satisfies all conditions on subsolutions except condition \ref{enum:subsol_ineq}.
    
If, moreover, $I_{\tau}(\mathbf{u})=0$ for every $\tau \in(0,T)$ then 
    $\frac{1}{2}|\mathbf{u}|^2=\overbar{e}[\mathbf{u}]=\frac{3}{2}\left(    Z(t)-\Pi[\mathbf{u}] \right)$ everywhere in $(0,T)$ and a.e. in $U$ 
    and \eqref{eq:sym_matr_eq} hold almost everywhere in $(0,T)\times \Omega$. Hence, 
thanks to the hypothesis of the lemma 
    \[  
        \BV = \BH[\mathbf{u}]+\mathbf{u}\otimes\mathbf{u}  -\frac{1}{3}|\mathbf{u}|^2\BI \quad \mbox{a.\,e.\,in } (0,T)\times \Omega
    \]
    and $\mathbf{u}$ is a weak solution of \eqref{eq:nonlin_separ}.
\end{proof}

\section{Existence result for the abstract Euler-type system}
\label{sec:existence}
An important step on the way to find $\mathbf{u}\in X$ satisfying \eqref{eq:maximal_energy} is the following possibility to appropriately perturb any subsolution so that $I$ increases. 
The proof of the following lemma is postponed until Section \ref{sec:osc_lem_proof}.
\begin{lem}[Oscillatory lemma]
    \label{lem:oscillatory}
    Let $\mathbf{u}\in X_0$ and $(\mathbf{u},\BV)$ be a subsolution and $\tau>0$. Then there exist sequences 
    $\{\mathbf{w}_n\}_{n\in \BN}\subseteq \CD((\tau,T)\times U;\BR^3)$ and 
    $\{\BW_n\}_{n\in \BN}\subseteq \CD((\tau,T)\times U;\BR^{3\times 3}_{0,sym})$ 
    such that:
    \begin{itemize}
        \item $(\mathbf{u}+\mathbf{w}_n,\BV+\BW_n)$ are subsolutions,
        \item $\mathbf{w}_n \to 0$ in $\CC([0,T];L_{w}^2(U))$,
        \item there exists $c=c(E)>0$ such that 
            \begin{equation}
                \label{eq:energy_increase}
                \liminf_{n\to \infty} I_{\tau}(\mathbf{u}+\mathbf{w}_n)\geq I_{\tau}(\mathbf{u})+c(E)\left(I_{\tau}(\mathbf{u})\right)^2.
            \end{equation}
    \end{itemize}
\end{lem}
\begin{rem}
The constant $c(E)$ does not depend on $\mathbf{u}$ or $\tau$.
\end{rem}

The existence of infinitely many weak solutions is then concluded from a Baire category argument similar to e.\,g.\,\cite{DeleSzekAnnals}.
\begin{lem}
    \label{lem:category}
    Let $(X,d)$ be a complete metric space, $I\colon X \to (-\infty;0]$ a function of Baire class $1$.
    Let $X_0$ be a nonempty dense subset of $X$ with the following property: for any $\beta<0$ there exists $\alpha=\alpha(\beta)>0$ such that for any $x\in X_0$ 
satisfying $I(x)<\beta<0$ there exists $x_n\in X_0$ with 
    \begin{itemize}
        \item $x_n\to x$ in $(X,d)$ and
        \item $\liminf_{n\to \infty} I(x_n)\geq I(x)+\alpha(\beta)$.
    \end{itemize}
    Then there exists a residual set $S\subseteq X$ such that $I(x)=0$ on $S$.
\end{lem}
\begin{proof}
    As $(X,d)$ is complete, the set of points of continuity of functions of Baire class $1$ on $X$ is residual. 
    To complete the proof, it is sufficient to show 
that $I=0$ on the set of points of continuity.
 We prove that by contradiction. Let $x$ be a point of continuity of $I$ such that $I(x)<\beta<0$. Then
    from the density of $X_0$, there exists a sequence $\{x_n\}_{n\in \BN} \subseteq X_0$, converging to $x$ and $I(x_n)\to I(x)$. 
Without loss of generality, we may assume that $I(x_n)<\beta$. For each $n\in \BN$ there exists sequence $\{x_{n,k}\}_{k\in \BN}$ satisfying the conditions 
given by the hypothesis of the lemma. By a diagonal argument we can find a subsequence $\{x_{n,k(n)}\}_{n\in \BN}\subseteq X_0$ such that $x_{n,k(n)}\to x$ and 
    \[
      \liminf_{n\to \infty}  I(x_{n,k(n)})\leq I(x_n)+\alpha\left(\beta\right).
    \]
    This contradicts the assumption that $x$ is a point of continuity of $I$.
\end{proof}

\begin{proof} [Proof of Theorem \ref{thm:abstr_exist}.] 
Let $X_0$ be the set of subsolutions to the abstract Euler system. $X_0$ consists of functions 
$\mathbf{u}: [0,T] \rightarrow L^2(U)$ taking values in a bounded subset $Y$ of $L^2(U)$.
Hence $Y$ is metrizable with respect to the weak topology of $L^2$. 
Correspondingly, we consider the metric $d$ naturally defined on $\CC([0,T];Y)$ which induces a topology
equivalent to the topology of $\CC([0,T];Y)$ as a subset of $\CC([0,T];L^2_w(U))$.
We denote by $X$ the completion of $X_0$ in $\CC([0,T];L^2_w(U))$ with respect to the metric $d$. 
Obviously, $X$ is bounded in $L^{\infty}(Q;\BR^3)$. The set $X_0$ is non-empty as it makes no difficulty to check that $\mathbf{u}(t) = \mathbf{u}_0$ with $\BV=\mathbf{0}$ defines a subsolution.

For each $\tau \in (0,T\slash 2)$, $I_{\tau}$ can be extended on a lower-semicontinuous functional on $X$ and therefore is of Baire class $1$. 
Indeed, observe that $\overbar{\mathbf{e}}$ is continuous from $(X,d)$ to $\CC(\overbar{Q})$, hence, the semicontinuity
follows from the case when $\overbar{\mathbf{e}}$ is a constant function (see \cite[Lemma 5]{DeleSzekArch}). 

Finally, a combination of Lemma \ref{lem:oscillatory}, \ref{lem:category} and \ref{lem:subsolconv} implies the
existence of residual sets 
\[
  S_{\tau} = \{x\in X \colon I_{\tau}(x) = 0\}.
\] 
The set $S=\cap_{n=1}^{\infty}S_{\frac{1}{n}}$ is also residual, especially nonempty and of infinite cardinality.
Due to Lemma \ref{lem:subsolconv}, all functions in $S$ are  weak solutions to the abstract Euler problem with
$\mathbf{u}(0, \mathbf{x}) = \mathbf{u}_0$ and such that \eqref{eq:energy constraint} holds. 
\end{proof}

\section{Proof of Lemma \ref{lem:oscillatory}}
\label{sec:osc_lem_proof}
We start with a special case of the oscillatory lemma when the operators $\BH$ and $\overbar{e}$ are not depending on $\mathbf{u}$. 
Let us define $\tilde{e}\colon \BR^3\times \BR^{3\times 3}_{0,sym}\times \BR^{3\times 3}_{0,sym} \to \BR$ by
\begin{equation*}
    \tilde{e}(\mathbf{u}, \BV, \BG) = \frac{3}{2}\lambda_{\max}\left(\mathbf{u}\otimes\mathbf{u} +\BG - \BV  \right).
\end{equation*}
For $f\in L^{\infty}\cap \CC(Q)$ and $\BG \in L^{\infty}\cap\CC(Q;\BR^{3\times 3}_{0,sym})$ we denote $X_{0,\BG,f}$ the set of all functions $\mathbf{u}$ satisfying
\begin{enumerate}
    \item $\mathbf{u}\in  C([0,T];L_{w}^2(U;\BR^3))\cap \CC((0,T)\times U;\BR^3)$,
    \item  exists $\BV\in L^{\infty}\cap \CC((0,T)\times U;\BR^{3\times 3}_{0,sym})$ 
    such that the pair $(\mathbf{u}, \BV)$ satisfies \eqref{eq:linear} in the sense of distributions on $(0,T)\times U$ and $\mathbf{u}\chi_{U}$ solves \eqref{eq:incompress} in the sense of distributions on $(0,T)\times \BR^3$,
    \item $\mathbf{u}(0)=\mathbf{u}_0$,
    \item \label{enum:cont_subsol_ineq} 
  for every $\tau>0$
    $\inf_{t\in (\tau,T),\,\mathbf{x}\in U}\left( f(t,\mathbf{x})-\tilde{e}(\mathbf{u}(t,\mathbf{x}),\BV(t,\mathbf{x}),\BG(t,\mathbf{x}))    \right)>0.$

\end{enumerate}
The following auxiliary result was proven in \cite{DonaFeirMarc} for $I_{\tau}$ defined using integrals with respect to time and space. 
In our case, the proof remains the same and we will omit it (see also \cite{DeleSzekArch} where the functional setting is the same as ours). 
\begin{lem}
    \label{lem:oscill1}
    Let $O=(\tau_1,\tau_2)\times U\subseteq Q$ be an open set, $\BG$ and $f$ be as above with  $f>0$ in $Q$.
    Assume that $\mathbf{u}\in X_{0,f,\BG}$. Then exists $\Lambda>0$ and sequences $\{\mathbf{w}_n\}_{n\in \BN}\subseteq \CD(O;\BR^3)$ and
    $\{\BW_n\}_{n\in \BN}\subseteq \CD(O;\BR^{3\times 3}_{0,sym})$ such that $(\mathbf{u}+\mathbf{w}_n,\BU+\BW_n)\in X_{0,f,\BG}$, 
    \[
        \mathbf{w}_n \to 0 \mbox{ in } \CC([0,T];L^2_w(U))
    \]
    and
    \begin{equation}
        \label{eq:fixed_oscillations}
        \liminf_{n\to \infty}\inf_{t\in(\tau_1,\tau_2)}\int_{U} |\mathbf{u}
        +\mathbf{w}_n|^2\de{x}\geq \inf_{t\in(\tau_1,\tau_2)}\int_{U} |\mathbf{u}|^2\de{\mathbf{x}}+\Lambda
        \left(
            \inf_{t\in(\tau_1,\tau_2)}\int_{U} f-\frac{1}{2}|\mathbf{u}|^2    \de{\mathbf{x}}
        \right)^2,
    \end{equation}
    where $\Lambda = \Lambda(\sup_{(t,\mathbf{x})\in Q}|f|)$ (namely, $\Lambda$ does not depend on $O$ or $\mathbf{u}_n$).
\end{lem}
Let us show that Lemma \ref{lem:oscillatory} follows from Lemma \ref{lem:oscill1} using a perturbation argument. 

\begin{proof}[Proof of Lemma \ref{lem:oscillatory}.]
If $\mathbf{u}\in X_0$ then there exists an increasing continuous function $\delta\colon (0,T)\to (0,+\infty)$
such that for any $s\in (0,T)$
\[
\inf_{t\in (s,T),\,x\in U}\left(\overbar{e}[\mathbf{u}]-e(\mathbf{u},\BV)    \right)>\delta(s).
\] 
Hence, $\mathbf{u}\in X_{0,\overbar{e}[\mathbf{u}]-\delta, \BH[\mathbf{u}]}$ and we obtain  sequences $\{\mathbf{w}_n\}_{n\in \BN}$ and $\{\BW_n\}_{n\in \BN}$ satisfying Lemma \ref{lem:oscill1} 
with $f=\overbar{e}[\mathbf{u}]-\delta$ and $\BG = \BH[\mathbf{u}]$.
Moreover, due to the boundedness of $\mathbf{w}_n$ and $\BW_n$, we have 
\[
      \mathbf{w}_n \to 0 \mbox{ in } \CC([0,T];L^p_w(U)) \quad \mbox{for any } p\in [1,\infty),
\]
see Lemma \ref{lem:subsol_limit}.
Inequality \eqref{eq:energy_increase} follows directly from \eqref{eq:fixed_oscillations} as $\overbar{e}[\mathbf{u}+\mathbf{w}_n]\to \overbar{e}[\mathbf{u}]$ uniformly in $Q$. 
Hence, to finish the proof, it is sufficient to check that $\mathbf{u}+\mathbf{w}_n \in X_0$ at least for  indices large enough. As $\mathbf{u}+\mathbf{w}_n\in X_{0,\overbar{e}[\mathbf{u}]-\delta, \BH[\mathbf{u}]}$, we get
\[
    e(\mathbf{u}+\mathbf{w}_n,\BV+\BW_n)=
    \tilde{e}(\mathbf{u}+\mathbf{w}_n,\BV+\BW_n,\BH[\mathbf{u}])
    +r_n< \overbar{e}[\mathbf{u}]-\delta + r_n
    =\overbar{e}[\mathbf{u+w_n}]-\delta + r_n + t_n,
\]
where
\[
    r_n=
       \tilde{e}(\mathbf{u}+\mathbf{w}_n,\BV+\BW_n,\BH[\mathbf{u}+\mathbf{w}_n])- \tilde{e}(\mathbf{u}+\mathbf{w}_n,\BV+\BW_n,\BH[\mathbf{u}])
\]
and
\[
    t_n = \overbar{e}[\mathbf{u}]-\overbar{e}[\mathbf{u+w_n}].
\]
The function $\BA\mapsto \lambda_{max}(\BA)$ restricted on the symmetric positive semidefinite matrices is equal to $\ell^{2}\to \ell^2$ operator norm, hence it is $1$--Lipschitz. Thus, using the continuity of  $\BH$ and $\overbar{e}$, we obtain 
\[
    r_n+t_n \to 0 \mbox{ uniformly in } [\tau,T]\times U.
\] 
Having in mind \eqref{eq:nonanticip}, we claim that $r_n+t_n = 0$ for $(t,\mathbf{x})\in (0,\tau)\times U$, therefore
\[
    e(\mathbf{u}+\mathbf{w}_n,\BV+\BW_n)<\overbar{e}[\mathbf{u+w_n}]-\frac{\delta}{2}
\]
 holds on $Q$ for sufficiently large $n$.
\end{proof}

\section{Dissipative solutions}
\label{sec:dissipative}
This Section is devoted to the proofs of Theorems \ref{thm:bous_main2} and \ref{thm:main2}.
Thanks to the reformulation of the Boussinesq and  the extended primitive equations in the framework of abstract Euler--type systems carried out in Section \ref{sec:reformulation}, 
Theorems \ref{thm:bous_main2} and \ref{thm:main2} can be reduced to prove the following more general theorem on the abstract system.
\begin{thm}
    \label{thm:general dissipative}
    Under the same hypotheses on $\BH$ and $\Pi$ of Theorem \ref{thm:abstr_exist}, there exists $\mathbf{u}_0\in L^{\infty}(U;\BR^3)$ for which we can find 
    infinitely many weak solutions to \eqref{eq:abs} emanating from $\mathbf{u}_0$ and such that the functional  
\begin{equation} \label{eq:energy constraint general}
E_{abs}(t):= \int_U \left(\frac{3}{2} \Pi[\mathbf{u}](t, \mathbf{x})+\frac{1}{2}|\mathbf{u}(t, \mathbf{x})|^2\right) \de{\mathbf{x}} \quad\text{is non--increasing on}\; [0,T).
\end{equation}
\end{thm}
\begin{rem}
Thanks to Theorem \ref{thm:abstr_exist}, and in particular to the property \eqref{eq:energy constraint}, the conclusion that the functional $E_{abs}(t)$ is non--increasing on $(0,T)$ 
can be achieved for any $\mathbf{u}_0 \in L^{\infty}(U;\BR^3)\cap \CC(U;\BR^3)$ 
    with $\diver_\mathbf{x} \mathbf({u}_0\chi_{U}) =0$ by simply choosing the function $Z(t)$ to be non--increasing on $(0,T)$. 
But in order to obtain dissipative solutions for the Boussinesq and primitive equations the property \eqref{eq:energy constraint general} is required up to time $t=0$: this forces the construction
of suitable initial data $\mathbf{u}_0\in L^{\infty}(U;\BR^3)$.
\end{rem}

We now show how Theorems \ref{thm:bous_main2} and \ref{thm:main2} follow from Theorem \ref{thm:general dissipative}.
\begin{proof}[Proofs of Theorems \ref{thm:bous_main2} and \ref{thm:main2}]
Due to the reformulations of the Boussinesq and extended primitive equations of Section \ref{sec:reformulation},
the respective choices for $\Pi$ are $\Pi_{Bous}[\mathbf{u}] = \frac{2}{3}z\Theta[\mathbf{u}]$ and $\Pi_{Prim}=0$.
These choices allow to obtain from $E_{abs}$ exactly $E_{Bous}$ and $E_{Prim}$ respectively. Hence the conclusion of Theorem \ref{thm:general dissipative} implies the existence
of infinitely many dissipative solutions to the Boussinesq and extended primitive equations starting from suitably constructed initial data $\mathbf{u}_0$, as stated 
in Theorems \ref{thm:bous_main2} and \ref{thm:main2}.
\end{proof}

\subsection{Construction of initial data}
The abstract Euler system \eqref{eq:abs} fits the framework 
introduced by Feireisl in \cite{FeirConvexInt}. In particular, we can apply Theorem 6.1 therein to obtain strong continuity in $L^2$ at time
$t=0$. For the sake of completeness, we report here a version of \cite[Theorem 6.1]{FeirConvexInt} adapted to our context. For other variants of the following result 
we refer to \cite{DeleSzekArch} and also to \cite{Chiodaroli}, \cite{ChioFeirKrem}.

\begin{lem}\label{thm:strong continuity}
 Let $\BH$ and $\Pi$ satisfy the hypotheses of Theorem \ref{thm:abstr_exist}. Then there exist a set of times $\mathcal{R}\subset (0,T)$ dense in $(0,T)$ such that 
for any $\tau \in \mathcal{R}$ there is $\mathbf{u}\in X$ with the following properties
\begin{itemize}
 \item[(i)]$\mathbf{u} \in C( ((0, \tau) \cup (\tau , T)) \times U)\cap C([0, T], L^2_w)$,  $\mathbf{u}(0, \cdot)= \mathbf{0}$,
 \item[(ii)]there exists $\BV  \in  C( ((0, \tau) \cup (\tau , T)) \times U; \BR^{3\times 3}_{0,sym})$ the pair $(\mathbf{u}, \BV)$ satisfies
 \eqref{eq:linear} in the sense of distributions on $Q$ and $\mathbf{u}\chi_U$ solves \eqref{eq:incompress} in the sense of distributions on $(0,T)\times \BR^3$,
 \item[(iii)] $\left( \overbar{e}[\mathbf{u}](t,\mathbf{x})-e(\mathbf{u}(t,\mathbf{x}),\BV(t,\mathbf{x}))    \right)>0$ for all $(t,\mathbf{x})\in ((0, \tau)\cup(\tau , T)) \times U$,
 \item[(iv)]$\frac{1}{2} |\mathbf{u}(\tau, \mathbf{x})|^2= \overbar{e}[\mathbf{u}](\tau,\mathbf{x})$ a.e. in $U$
\end{itemize}
where we recall that 
\[
    \overbar{e}[\mathbf{u}](t,\mathbf{x}) = \frac{3}{2}\left( Z(t)-\Pi[\mathbf{u}](t,\mathbf{x}) \right) 
\]
for a continuous function $Z(t)$ satisfying $\sup_{t\in [0,T]} Z(t)>\overbar{\Pi}$.
\end{lem}
\begin{rem}
Lemma \ref{thm:strong continuity} provides subsolutions which are strongly continuous at the point $\tau$ and allow to obtain the desired strong energy conditions.
\end{rem}
For the proof we refer the reader to \cite{FeirConvexInt}.

\begin{proof}[Proof of Theorem \ref{thm:general dissipative}]
The proof consists in finding an initial datum $\mathbf{u}_0\in L^{\infty}(U;\BR^3)$ and a function $Z(t)$ with the following properties
\begin{equation} \label{eq:energy in 0}
\frac{1}{2} |\mathbf{u}_0 (\mathbf{x})|^2= \overbar{e}[\mathbf{u}_0](\mathbf{x}) \quad \text{a.e. in}\; U; 
\end{equation}
\begin{itemize}
 \item $Z(t)$ continuous on $[0,T]$ and $\sup_{t\in [0,T]} Z(t)>\overbar{\Pi}$,
 \item $Z'(t)\leq 0$ for all $t\in [0, T)$
\end{itemize}
and such that the set $X_0$ of subsolutions $\mathbf{u}$ associated to this datum is non--empty.
First of all, we notice that we can easily choose $Z(t)= C_Z$ for some constant $C_Z>\overbar{\Pi}$.
Once $Z$ has been chosen, then we apply Lemma \ref{thm:strong continuity} which provides the existence of 
a time $\tau\in \mathcal{R}$ and a function $\mathbf{u}$ for which (i)--(iv) hold.
We define the initial datum $\mathbf{u}_0$ to be $\mathbf{u}_0( \cdot) = \mathbf{u}(\tau, \cdot)$ in $U$.
To such a datum we associate, as in Section \ref{sec:reformulation}, the set of subsolutions $X_0$ and we can prove that it is non--empty by choosing as eligible element
the following subsolution
\begin{equation*}
 \bar{\mathbf{u}}(t, \mathbf{x})= \begin{cases}
                                    \mathbf{u}(t+\tau, \mathbf{x})  \quad &\text{for} \, t\in [0, T- \tau]\\
                                    \mathbf{u}(t- (T-  \tau), \mathbf{x})  \quad &\text{for} \, t\in [ T- \tau, T]
                                  \end{cases}
\end{equation*}
with relative matrix field $\bar{\BV}$ analogously defined. 
Indeed from \eqref{eq:energy in 0}, redoing the proof of Theorem \ref{thm:abstr_exist}, we would now obtain infinitely many solutions to \eqref{eq:abs} emanating from $\mathbf{u}_0$ 
and such that
\begin{equation*} 
 \left(\frac{3}{2} \Pi[\mathbf{u}](t, \mathbf{x})+\frac{1}{2}|\mathbf{u}(t, \mathbf{x})|^2\right) = \frac{3}{2} C_Z \quad \text{for all} \; t\in [0,T) \; \text{and a.e. in}\; U
 \end{equation*}
(we remark that the equality now holds up to time $t=0$) which implies Theorem \ref{thm:general dissipative}.
 
\end{proof}

\section{Appendix}
\label{sec:appendix}
\subsection{The Lam\'e system}
Let us denote 
\[
    \BD_0(\mathbf{v}) = \left(\nabla \mathbf{v} + \nabla \mathbf{v}^T- \frac{2}{3}\diver_\mathbf{x} \mathbf{v} \BI\right)
\]
and let $\mathbf{g}\colon U \to \BR^3$. The Dirichlet boundary value problem for the Lam\'e system is a  
whether there exists a function $\mathbf{v} \colon U \to \BR^3$ with zero trace on $\partial \Omega$ such that
\begin{equation}
    \label{eq:lame}
    \diver_{\mathbf{x}}\BD_0(\mathbf{v}) =\mathbf{g}.
\end{equation}

\begin{lem}
    Let $U\subseteq \BR^3$, $\partial U\in \CC^2$, $\mathbf{g}\in L^p(U,\BR^3)$ and $p\in(1,\infty)$. Then there exists a unique 
    $\mathbf{v}\in W^{2,p}(U,\BR^3)$ with zero trace satisfying \eqref{eq:lame} almost everywhere in $U$ and the operator 
    $\mathbf{g} \mapsto \mathbf{v} \colon L^p(U;\BR^3)\to W^{2,p}(U;\BR^3)$ is continuous.
\end{lem}
\begin{proof}
  We only show that the elliptic operator in \eqref{eq:lame} satisfies the Legendre-Hadamard conditions. 
As the operator has constant coefficients, the existence, uniqueness and regularity
follows directly from the standard theory of elliptic systems (see for example \cite{GiaqMart}).
Let us denote $A^{\alpha,\beta}_{i,j}$, where $\alpha$, $\beta$, $i$, $j\in \{1,2,3\}$ the coefficients of the elliptic system
\eqref{eq:lame} (for the notation check \cite{GiaqMart}).
Then
\begin{equation*}
  \sum_{\alpha, \beta, i, j=1}^{3}A^{\alpha, \beta}_{i,j}\xi_{\alpha}\xi_{\beta}\eta^{i}\eta^j
  = \xi\otimes \eta : \left(\xi\otimes \eta + (\xi\otimes \eta)^T - \frac{2}{3} \xi\cdot \eta \BI\right)
  = |\xi|^2 |\eta|^2+\frac{1}{3}|\xi\cdot \eta|^2 \geq |\xi|^2 |\eta|^2.
\end{equation*}
\end{proof}
\begin{cor}
    \label{cor:tosym}
    Let $U\subseteq \BR^3$, $\partial U\in \CC^2$ and $p\in(1,\infty)$. Then there exists a continuous
    operator $\BG \colon L^p(U,\BR^3) \to W^{1,p}(U; \BR^{3\times3}_{sym,0})$ such that
    \[
        \diver_\mathbf{x} (\BG[\mathbf{g}]))=\mathbf{g}.
    \]  
\end{cor}

\subsection{Parabolic regularity}
The standard regularity result for parabolic equations (see e.\,g. \cite{amann}) gives
\begin{equation}
    \label{eq:embToCont}
    W^{1,q}((0,T);L^q(U))\cap L^q((0,T);W^{2,q}(U)\cap W^{1,q}_0(U))\hookrightarrow \CC([0,T];W^{1,q}_0(U))\hookrightarrow \CC([0,T]\times \overbar{U})
\end{equation}
whenever $q>3$.
Therefore we have:
\begin{lem}
    \label{lem:parab_regul}
    Assume that $q\in(3,\infty)$.
    Let $\theta_0\in W^{1,q}_0(U)$ and $\mathbf{u} \in L^{\infty}(Q;\BR^3)$. Then exists a unique
    \[
        \theta\in W^{1,q}((0,T);L^q(U))\cap L^q((0,T);W^{2,q}(U)\cap W^{1,q}_0(U))
    \] 
    which satisfies \eqref{eq:temp} almost everywhere and $\theta(0)=\theta_0$. Moreover, the operator $u \mapsto \Theta[\mathbf{u}]$ is continuous from $L^{\infty}(Q)$ to $\CC([0,T]\times \overbar{U})$ and the comparison principle holds, i.\,e. for two solutions $\theta^1$, $\theta^2$ emanating from $\theta^1_0$, $\theta^2_0$ holds that
    \[
        \mbox{if } \theta^1_0\leq \theta^{2}_0 \mbox{ a.\,e. in $U$ then } \theta^1(t,\mathbf{x})\leq \theta^{2}(t,\mathbf{x})
        \mbox{ a.\,e. in $Q$}.
    \]
    Moreover, 
    \[
         \|\theta\|_{W^{1,q}((0,T);L^q(U))\cap L^q((0,T);W^{2,q}(U)\cap W^{1,q}_0(U))}
         \leq C(\|\theta_0\|_{W^{1,q}}+\|\mathbf{u}\|_{L^{\infty}}),
    \]
    therefore given $\theta_0 \in W^{1,q}$, the solving operator
    \[
      \mathbf{u} \mapsto \Theta [\mathbf{u}]
    \]
    is continuous from $\CC([0,T];L^{q}_{w}(\Omega)$ to 
    $\CC([0,T]\times \overbar{U})$.
\end{lem}

\subsection{Convergence in linear conservation laws}
For the reader's convenience, we also recall the following standard weak compactness result for linear conservation laws.
\begin{lem}
    \label{lem:subsol_limit}
    Let $\{\mathbf{u}_{n,0}\}_{n\in \BN}$ converges weakly-$*$ in $L^{\infty}(U;\BR^3)$ to $\mathbf{u}_{0}$.  
    Let $\{\mathbf{u}_n\}_{n\in \BN}$ be a bounded sequence in $(L^{\infty}(Q;\BR^3))$
    and $\{\BV_n\}_{n\in \BN}$ be a bounded sequence in $(L^{\infty}(Q;\BR^{3\times 3}_{0,sym}))$
    satisfying
    \begin{align}
        \label{eq:linmom}
        \partial_t \mathbf{u}_n + \diver_\mathbf{x} \BV_n = 0 &\quad \mbox{in } \CD'(Q;\BR^3),
        \\
        \label{eq:conteq}
        \diver_\mathbf{x} (\mathbf{u}_n\chi_U)=0 &\quad \mbox{in } \CD'(\BR^3),
        \\
        \mathbf{u}_{n}(0)=\mathbf{u}_{n,0}.
    \end{align}
    Then $\{\mathbf{u}_n\}_{n\in \BN}$ is precompact in $\CC([0,T];L_w^p(U))$ for every $p\in[1,\infty)$. 
    Moreover, if $(\mathbf{u},\BV)$ is a limit of any weakly-$*$ convergent subsequence of $\{(\mathbf{u}_n,\BV_n)\}_{n\in \BN}$ then $(\mathbf{u},\BV)$ satisfies \eqref{eq:linmom}, \eqref{eq:conteq} and $\mathbf{u}(0,\cdot)=\mathbf{u}_0(\cdot)$.
\end{lem}

\renewcommand{\abstractname}{Acknowledgements}
\begin{abstract}
We are grateful to professor E.\,Feireisl for pointing out the problem and for fruitful
discussions.
The research of M.M. leading to these results has received funding from the European Research Council 
under the European Union's Seventh Framework Programme (FP7/2007-2013)/ ERC Grant Agreement 320078. 
The Institute of Mathematics of the Academy of Sciences of the Czech Republic is supported by RVO:67985840. 
\end{abstract}

\bibliographystyle{plain}

{\footnotesize
\bibliography{references/reference}
}


\end{document}